\documentclass[11pt,fleqn]{amsart}
\usepackage{latexsym,amsmath,amssymb,amsfonts,MnSymbol,epsfig}
\usepackage[initials]{amsrefs}
\usepackage[all]{xy}

\providecommand{\lra}{\longrightarrow}

\providecommand{\CC}{{\mathbb{C}}}
\providecommand{\RR}{{\mathbb{R}}}

\providecommand{\HH}{{\mathcal H}}
\providecommand{\KK}{{\mathcal K}}
\providecommand{\LL}{{\mathcal L}}

\newcommand{\ang}[1]{\langle #1 \rangle} 

\newtheorem{definition}{Definition}
\newtheorem{lemma}[definition]{Lemma}

\newtheorem{corollary}[definition]{Corollary}
\newtheorem{theorem}[definition]{Theorem}

\begin{document}

\title{An Index Formula for the Extended Heisenberg Algebra of Epstein, Melrose and Mendoza}
\author{Erik van Erp}
\address{Department of Mathematics, Dartmouth College, 27 N. Main St., 
Hanover, New Hampshire, USA}
\email{jhamvanerp@gmail.com}

\begin{abstract}
The extended Heisenberg algebra for a contact manifold has a symbolic calculus that accommodates both Heisenberg pseudodifferential operators as well as classical pseudodifferential operators. 
We derive here a formula for the index of Fredholm operators in this extended calculus. 
This formula incorporates in a single expression the Atiyah-Singer formula for elliptic operators, as well as Boutet de Monvel's Toeplitz index formula. 
\end{abstract}

\maketitle

\section{Introduction}

In \cite{Me97} Epstein, Melrose and Mendoza introduce the {\em extended Heisenberg algebra} for a closed contact manifold. 
Their extended algebra contains, as subalgebras, both the classical pseudodifferential operators as well as operators in the Heisenberg calculus,
and is, more or less, generated by these two subalgebras.
(A similar algebra had been defined by Taylor in \cite{Ta84}.)
Classical elliptic operators as well as certain hypoelliptic operators associated to the contact structure (like Toeplitz operators) have an invertible symbol in the extended calculus. 

In \cite[section 8]{Me97} a solution to the Fredholm index problem for the extended Heisenberg calculus is sketched out.
Details of the proposed solution were worked out in the manuscript \cite{Epxx}, but the results were not quite satisfying,
and the manuscript \cite{Epxx} remains, unfortunately, unpublished.
The resulting explicit formula is more complicated than the brief remarks in \cite{Me97} would lead one to believe.

In \citelist{\cite{vE10a} \cite{vE10b}} we presented a different approach to this index problem. 
We obtained an explicit formula for the Fredholm index in the Heisenberg calculus proper that was certainly much simpler than the final formula arrived at in \cite{Epxx}.
Our results (based on an adaptation of  Connes' tangent groupoid approach) were arguably more satisfactory than those achieved in \cite{Epxx}.

In the present paper we take a fresh look at this problem.
We are now able to generalize our explicit formula \cite[Theorem~6]{vE10b} from the Heisenberg algebra to the extended Heisenberg algebra.
As indicated in \cite{Me97}, the resulting formula unites, in a single expression, the index formulas of Atiyah-Singer (for standard elliptic pseudodifferential operators) and the Toeplitz index theorem of Boutet de Monvel.

Our approach here does not rely on the tangent groupoid.
Our method here is closer in spirit to what was proposed originally in \cite{Me97}.
In particular, we show here how the index problem for the extended Heisenberg algebra can be reduced (in two steps) to Boutet de Monvel's theorem.
A key intermediate step is the reduction from the extended Heisenberg algebra to the ideal of Hermite operators.
The reduction of the Hermite problem to the Toeplitz problem was already worked out, in detail, in \cite{Epxx}. 
A basic simplification over the method used in \cite{Epxx} is obtained if we work in the $C^*$-algebra formalism, because now symbols are only required to be continuous. 
The Hermite ideal, in our approach, is simply defined by symbols that reduce to the identity on the equator, rather than having to osculate the identity to high order, as in the work of Epstein and Melrose.
As we will see below, the reduction of the Hermite index problem to the Toeplitz index problem becomes, in the $C^*$-algebraic formalism, a straightforward argument in analytic $K$-homology. 

The more interesting problem is how to reduce the full problem to the Hermite problem.
We achieve this by means of a crucial trick involving the transposition of operators.
This use of transposition is, to our knowledge, a novel idea in the context of index theory.
The resulting formula, Theorem \ref{theorem}, is remarkably straightforward, especially when compared with the formulas in \cite{Epxx}.
A curious feature of our method here is that it does {\em not} extend to operators acting on vector bundles.

\section{Pseudodifferential calculi}\label{section:two}

Our primary reference for the extended Heisenberg calculus are \cite{Me97} and \cite{Epxx}.
For convenience of the reader, we provide in this section a brief sketch of the main features of this calculus.
We also consider the structure theory of the $C^*$-algebra that is the closure of its symbol algebra.

\vskip 6pt

A pseudodifferential operator on a smooth manifold $M$ (in any calculus) is pseudolocal, which means that its Schwartz kernel $k(x, y)$ is smooth ($C^\infty$) away from the diagonal $D=\{(x,x)\in M\times M\}$. 
In other words, the singular support of the Schwartz kernel $k$ of a pseudodifferential operator is contained in $D$. 
On a closed manifold there is no issue with the `proper support' of the kernel $k$. 
Therefore, to specify a pseudodifferential calculus on a closed manifold amounts to finding a way to control the asymptotics of the singularity of $k$ in the direction transversal to the diagonal.
It is convenient to describe these asymptotics in terms of the Fourier transform of the kernel $k$ in transverse directions.

In what follows $M$ is a smooth closed manifold with a contact hyperplane bundle $H\subset TM$.
We denote by $N=TM/H$ the quotient line bundle, and assume that it is a trivial line bundle.

Let $U$ be a neighborhood of the diagonal $D$ in $M\times M$.
Choose connections $\nabla^H$ and $\nabla^N$ in $H$ and $N$ respectively.
We pick a section $N\subset TM$, so that $TM=H\oplus N$.
Then let $\nabla=\nabla^H+\nabla^N$ be a connection on $TM$.
We are interested in the exponential map ${\exp}^\nabla$ associated to this connection $\nabla$,
\[ h\;\colon\; T_xM\to M\times M\;\colon\; (x, v)\mapsto ({\rm exp}^\nabla_x(v/2), {\rm exp}^\nabla_x(-v/2)).\]
Let $\phi\in C_c^\infty(U)$ be a function supported in $U$ that is equal to $1$ in a neigborhood of $D$.
We use the map $h$ to pull back the distribution $\phi k$ to $TM$.
The Fourier transform
\[ a(x, \xi) = \int_{y\in T_xM} e^{-i\ang{y,\,\xi}} \,h_*(\phi k(x,y))\]
is the {\em amplitude} of an oscillatory integral operator whose Schwartz kernel agrees with $k$ up to a smoothing operator.
The amplitude $a(x,\xi)$ is a function on the cotangent bundle $T^*M$,
and the asymptotics of the kernel $k(x,y)$ near the diagonal $D$ are reflected in the asymptotics at infinity of the amplitude.

Let $D^*M$ denote the co-disk bundle of $M$, i.e., the radial compactification of $T^*M$.
The space $D^*M$ is a smooth manifold with boundary.
A defining function for its boundary is given by
\[ \rho_R(\xi) = (1+\|\xi\|^2)^{-1/2}.\]
The function $\rho_R$ defines a transversal coordinate (with values in $[0, \infty)$) in a neighborhood of the boundary $S^*M$ of $D^*M$,
such that $\rho_R^{-1}(0)$ is precisely the boundary $S^*M$.
Moreover, the coordinate $\rho_R$ can be completed to a full chart (with values in a half-space $[0,\infty)\times \RR^N$) near any point on the boundary.
In this way the function $\rho_R$ defines the structure of $D^*M$ as a smooth manifold with boundary. 

An amplitude $a(x, \xi)$ corresponds to a {\em classical} pseudodifferential operator of order $m$ precisely if $\rho_R(\xi)^{m} a(x,\xi)$ extends to a smooth function on $D^*M$.
The restriction of $\rho_R(\xi)^{m} a(x,\xi)$ to the boundary $S^*M$ is the principal symbol of the operator,
which is well-defined (independent of the choice of metric or cut-off function $\phi$) as a smooth function on $S^*M$.
The principal symbol corresponds to the leading term in the Taylor expansion of $a(x,\xi)$ near the boundary $S^*M$.
If $P$ is a classical pseudodifferential operator of order $m$,
then we denote its principal symbol by $\sigma^m(P)$.
\vskip 6pt

The Heisenberg calculus can be obtained in a similar way, by choosing a different compactication of $T^*M$.
Let $H\subset TM$ denote a contact hyperplane bundle on $M$.
The dual $N^*$ of the normal bundle $N=TM/H$ is a line subbundle of $T^*M$.
Choose a splitting $T^*M\cong H^*\oplus N^*$,
and denote by $\xi=(\xi_H, \xi_N)$ the corresponding decomposition of covectors.
Let $\rho_H$ be the smooth function on $T^*M$ defined by
\[ \rho_H(\xi) = (1+\|\xi_H\|^4+\|\xi_N\|^2)^{-1/4}.\]
Just as before, the function $\rho_H$ defines a compactification for $T^*M$,
and providing it with the structure of a smooth manifold with boundary.
It may not be immediately evident, but the result is independent of the choice of section $H^*\subset T^*M$.
We denote the resulting compactification by $D^*_HM$, and its boundary by $S^*_HM$.
As a smooth manifold with boundary $D^*_HM$ is diffeomorphic to $D^*M$,
but not naturally.
Epstein, Melrose and Mendoza refer to this compactification as the {\em parabolic compactification} of $T^*M$.
The name is derived from the fact that the parabolic cosphere bundle $S^*_HM$ can be identified with the manifold of parabolic rays of the form $\{(t\xi_H, t^2 \xi_N), t>0\}$ in $T^*M$,
just as the radial boundary $S^*M$ is the manifold whose points are the `straight line' rays $\{(t\xi_H, t\xi_N), t>0\}$.

Now we proceed as before.
An operator is a pseudodifferential operator of order $m$ in the {\em Heisenberg calculus} precisely if its amplitude $\rho_H(\xi)^{m}a(x,\xi)$ extends to a smooth function on the parabolic compactification $D^*_HM$.
The restriction of $\rho_H(\xi)^{m}a(x,\xi)$ to the boundary sphere bundle $S^*_HM$ is independent of choices, and determines the principal symbol of the operator in the Heisenberg calculus.
If $P$ is a Heisenberg pseudodifferential operator of order $m$,
we denote its principal symbol in the Heisenberg calculus by $\sigma_H^m(P)$.
\vskip 6pt

Before we describe the algebraic structure of the symbolic calculus,
we need to describe some geometric structure on the parabolic cosphere bundle $S^*_HM$.
The intersection of the dual of the hyperplane bundle $H^*\subset D^*_HM$ with the boundary sphere $S^*_HM \cap H^*$ is the sphere bundle $S^*H$ of $H^*$.
We refer to the spheres in $S^*H$ as the {\em equatorial spheres} in the fibers of $S^*_HM$.
We assume that the contact manifold $(M, H)$ is {\em co-oriented},
so that the submanifold $S^*H$ divides $S^*_HM$ into two components.
We refer to the two components of $S^*_HM\setminus S^*H$ as the {\em upper} and {\em lower hemispheres},
and denote them by $S^*_{H,\pm}M$.
We have the disjoint union
\[ S^*_HM = S^*_{H,+}M \sqcup S^*_{H,-}M \sqcup S^*H.\]
Further structure is introduced if we choose a {\em contact  $1$-form} $\theta$ on $M$, compatible with the co-orientation.
Then $d\theta$ restricts to a symplectic form in the fibers of $H$,
which induces a symplectic form on $H^*$ as well.
Observe also that the choice of $\theta$ amounts to choosing a trivialization of $N^*$, i.e., we may identify $N^*_x=\RR$.
Each point in the upper hemisphere $S^*_{H, +}M_x$ (at a point $x\in M$) corresponds to a parabolic ray $\{(t\xi_H,t^2\xi_N), t>0\}$ in $T^*_xM$ with positive value $\xi_N>0$.
Each such ray intersects the hyperplane $H_x^*\times \{1\}$ in exactly one point. 
We thus have a natural diffeomorphism
\[ S^*_{H,+}M \cong H^*.\]
The same holds for the lower hemisphere.

\vskip 6pt
An integral part of any pseudodifferential theory is a symbolic calculus.
For classical pseudodifferential operators the principal symbol calculus is {\em commutative}.
If $P$ and $Q$ are two pseudodifferential operators of order $m, m'$,
then $PQ$ is a classical pseudodifferential operator of order $m+m'$,
and the principal symbols are multiplied pointwise as functions on $S^*M$,
\[ \sigma^{m+m'}(PQ) = \sigma^m(P) \sigma^{m'}(Q).\]
The Heisenberg symbolic calculus is more delicate.
A detailed description is found in \cite{Epxx}, but this manuscript is unpublished.
Other good references are available \citelist{\cite{BG88} \cite{CGGP92} \cite{Ta84}}.
We provide an overview.
If $\sigma_H(P)$ denotes the Heisenberg symbol of an operator $P$,
represented as a function on the parabolic sphere bundle $S^*_HM$,
then we denote by $\sigma_{H, +}(P)$ resp. $\sigma_{H, -}(P)$ the restriction of $\sigma_H(P)$ to the upper and lower hemisphers respectively.
The restriction of $\sigma_H(P)$ to the equator $S^*H$ is denoted by $\sigma_{H, 0}(P)$.
Observe that $\sigma_{H,0}$ is the common boundary value of $\sigma_{H,+}$ and $\sigma_{H,-}$.

The equatorial symbol $\sigma_{H,0}$ multiplies pointwise as a function on the cosphere bundle $S^*H$,
\[ \sigma^{m+m'}_{H,0}(PQ) = \sigma^m_{H,0}(P)\, \sigma^{m'}_{H,0}(Q).\]
But symbolic multiplication on the hemispheres is not commutative.
A key fact is that the upper and lower hemispheres are algebraically {\em independent}.
Multiplication on the upper hemisphere is denoted by $\#_+$,
while $\#_{-}$ denotes multiplication on the lower hemisphere,
\begin{align*}
\sigma^{m+m'}_{H,+}(PQ) &= \sigma^m_{H,+}(P)\,\#_+\, \sigma^{m'}_{H,+}(Q),\\
\sigma^{m+m'}_{H,-}(PQ) &= \sigma^m_{H,-}(P)\,\#_-\, \sigma^{m'}_{H,-}(Q).
\end{align*}
Identifying the hemispheres with $H^*$, as explained above, the sharp product $\#$ is the familiar composition formula for symbols in the {\em Weyl algebra} associated to a symplectic vector space,  
\[ a\#_{\pm} b (\xi) = \pi^{-2n}\int e^{\pm 2i\,d\theta(u, v)} a(\xi+u)b(\xi+v)dudv.\]
We shall not make explicit use of this formula,
but at this point it is useful to observe that any linear map in the fibers of $H^*$ that preserves the symplectic form $d\theta$ induces an {\em automorphism} of the symbol algebra. 
This fact will be crucial in the proof of our main theorem.

In the literature on the Heisenberg calculus it is not customary to consider the $C^*$-algebraic point of view, which is not useful when one is interested in regularity properties of operators. 
But for the purposes of index theory the $C^*$-algebraic completion of the various algebras is easier to understand and more convenient to work with.
We will therefore describe the appropriate $C^*$-algebraic completion $S_H$ of the Heisenberg symbol algebra.
Abstractly, this $C^*$algebra is the quotient of the norm-completion $\Psi_H$ of the algebra of order zero Heisenberg operators (realized as bounded operators on $L^2(M)$) by the ideal of compact operators $\KK$,
\[ 0\to \KK\to \Psi_H \to S_H \to 0.\]
This shows that there is a canonical $C^*$-norm on the algebra of Heisenberg symbols.
We have given an explicit description of this norm in \cite{vE10b}. 
To describe it, we start with the {\em Bargmann-Fok space} associated with the symplectic vector space $H^*_x$.
Choose a complex structure $J$ (with $J^2=-1$) in the fibers of $H^*$ that is compatible with the symplectic structure $d\theta$.
Let $H^{1,0}$ denote the $J=\sqrt{-1}$ eigenspace in the complexification $H\otimes \CC$.
The Bargmann-Fok space $V^{BF}_x$ at a point $x\in M$
is the Hilbert space of holomorphic functions $f$ on the complex vector space $H^{1,0}_x$ with norm $\|f\|^2 = \int |f|^2e^{-|z|^2}dz$.
Observe that $V^{BF}_x$ is the closure of the space of complex polynomials.
The family of Hilbert spaces $V^{BF}_x$ forms a continuous field of Hilbert spaces over $M$ (a Hilbert module over $C(M)$) which we denote by $V^{BF}$.

The algebra of Heisenberg symbols $\sigma_{H, +}$ (in the upper hemisphere)  
with its sharp product $\#_+$ has a natural representation on the Bargman-Fok space $V^{BF}_x$ (see \cite{Epxx}),
and the completion of the algebra of symbols in this norm is the {Weyl algebra} $\mathcal{W}_x$.
All we need to know for our purposes about this $C^*$-algebra $\mathcal{W}$ is that it fits in a long exact sequence
\[ 0\to\KK_x\to \mathcal{W}_x \to C(S^*_xH)\to 0.\]
The map $\mathcal{W}_x\to C(S^*_xH)$ is induced by the map which restricts a symbol $\sigma_{H,+}(P)$ in the upper hemisphere to its value at the boundary $S^*_xH$ (i.e., the equator).  
The kernel of this map completes to the $C^*$-algebra $\KK_x=\KK(V^{BF}_x)$ of compact operators on the Bargmann-Fok space $V^{BF}_x$.
The lower hemisphere gives rise to a similar situation.

The families of $C^*$-algebras $\KK_x$ and $\mathcal{W}_x$ for $x\in M$ form  continuous fields over $M$ in the obvious way.
We will denote the $C^*$-algebras of sections in these fields by $\KK(V^{BF})$ and $\mathcal{W}(V^{BF})$ respectively, to emphasize that they are realized concretely as families of operators on the Bargmann-Fok spaces.
We have the exact sequence
\[ 0\to \KK(V^{BF}) \to \mathcal{W}(V^{BF}) \stackrel{\sigma_0}{\lra} C(S^*H) \to 0.\]
The field $\{\KK_x,x\in M\}$ can of course be trivialized, since the Hilbert bundle $V^{BF}$ can be trivialized 
(like every Hilbert bundle with separably infinite dimensional fibers).
But this is not true for the field of Weyl algebras $\mathcal{W}_x$.

We have so far focused only on the upper hemisphere $S^*_{H,+}M$.
The $C^*$-algebra  $\mathcal{W}(V^{BF})$ is the completion of the algebra of symbols in the upper hemisphere.
The lower hemisphere gives an isomorphic $C^*$-algebra.
The two algebras are glued along the equator, and so the full Heisenberg symbol algebra completes to the $C^*$-algebra
\[ S_H = \{(a,b)\in \mathcal{W}(V^{BF})\oplus \mathcal{W}(V^{BF})\,\mid\, \sigma_0(a)=\sigma_0(b)\}.\]

\section{The extended Heisenberg calculus}\label{section:three}

The {extended Heisenberg algebra} is a pseudodifferential calculus introduced by Epstein, Melrose and Mendoza that contains {\em both} the classical and the Heisenberg pseudodifferential operators. 
It is obtained by choosing yet another compactification of $T^*M$.
The boundary $S^*M$ of the radial compactification $D^*M$
consists of straight rays $(t\xi_H, t\xi_N)$,
and the boundary $S_H^*M$ of the parabolic compactification $D^*_HM$
consists of parabolic rays $(t\xi_H, t^2\xi_N)$.
Every straight ray $(t\xi_H, t\xi_N)$ in $D^*_HM$ converges to the equatorial boundary in $S_H^*M$, unless $\xi_H=0$.
Likewise, every parabolic ray $(t\xi_H, t^2\xi_N)$ in $D^*M$ converges to one of the two points of intersection of $N^*$ with $S^*M$ (the two `poles' on the sphere), unless $\xi_N=0$.

This suggests an extended compactification of $T^*M$ whose boundary can be thought of as obtained from the radial boundary $S^*M$ by blowing up each of the two poles in the fibers to one of the (open) hemispheres of $S^*_HM$. 
Recall that each such hemisphere can be identified with $H^*_x$.
Alternatively, this extended boundary can be obtained from the parabolic boundary $S^*_HM$ by blowing up its equator $S^*H$ to $S^*M$ minus the poles.
The latter can be identified with $S^*H\times \RR$
(using a trivialization of $N^*$ by means of a contact form).
Let's denote this new boundary by $S^*_{eH}M$, and the corresponding compactification of $T^*M$ by $D^*_{eH}M$.
This extended compactification is a smooth manifold with corners.
We have a decomposition
\[ S^*_{eH}M = S^*_{H, +}M \sqcup S^*_{H,-}M \sqcup S^*H\times [-\infty, +\infty].\]
The corresponding pseudodifferential calculus is the {\em extended Heisenberg calculus}.
Operators in this calculus can have `mixed order'.
Let $\Psi_{eH}^{m,k}$ denote the set of extended Heisenberg operators of classical order $m$ (as measured by the growth rate of the amplitude $a(x,\xi)$ against the defining function $\rho_R$) and of Heisenberg order $k$ (now using $\rho_H$).
The principal symbol of such an operator we denote by $\sigma^{m,k}_{eH}$,
which is a smooth function on $S^*_{eH}M$.
This extended symbol can be restricted to three natural regions,
corresponding to the decomposition of $S^*_{eH}M$.
We denote the restriction of $\sigma_{eH}$ to the `classical' region $S^*H\times [-\infty, +\infty]$ by $\sigma_0$, the same notation we used before for the equatorial symbol. 
In the extended Heisenberg calculus the `classical' symbol $\sigma_0$ multiplies pointwise as a function on $S^*H\times [-\infty, +\infty]$.
The restrictions of the extended symbol $\sigma_{eH}$ to the Heisenberg region in $S^*_{eH}M$, i.e., the upper and lower hemispheres, we denote by $\sigma_{+}$ and $\sigma_{-}$, just as before. 
The symbols $\sigma_{\pm}$ multiply according to the sharp product $\#$, exactly  as in the Heisenberg calculus.
Observe that the values of $\sigma_{+}$ and $\sigma_{-}$ at the boundary $S^*H$ are equal to the values of $\sigma_{0}$ at its northern and southern boundaries, respectively.

The $C^*$-algebraic completion $S_{eH}$ of the extended Heisenberg symbol algebra
is very similar to the the symbol $C^*$-algebra $S_H$,
except that the equator is blown up.

Let $\Psi_{eH}$ denote the norm completion of $\Psi_{eH}^{0,0}$, the operators of order $(0,0)$ in the extended Heisenberg calculus.
The kernel of the principal symbol map $P\mapsto \sigma^{0,0}_{eH}(P)$
consists of the extended Heisenberg operators of order $(-1,-1)$.
Operators in $\Psi^{-1,-1}_{eH}$ are compact, 
while all smoothing operators are contained in $\Psi^{-1,-1}_{eH}$.
It follows, as usual, that the norm closure of $\Psi^{-1,-1}_{eH}$ is the $C^*$-algebra $\KK$ of compact operators on $L^2(M)$.
We therefore have the short exact sequence
\[ 0\to \KK\to \Psi_{eH} \stackrel{\sigma_{eH}}{\lra} S_{eH} \to 0,\]
where
\begin{align*}
S_{eH} = \{(a,b,c)&\in \mathcal{W}(V^{BF})\,\oplus\, \mathcal{W}(V^{BF})\,\oplus\, C(S^*H\times [-\infty,+\infty])\,\mid\, \\
&\sigma_0(a)=c(+\infty),\, \sigma_0(b)= c(-\infty)\}.
\end{align*}
Here $c(\pm\infty)\in C(S^*H)$ refers to the two boundary restrictions of the function $c\in C(S^*H\times [-\infty,+\infty])$.

The classical pseudodifferential operators are contained in this extended calculus.
Their extended symbol will simply have {\em constant} (i.e., scalar) values for 
$\sigma_{+}$ and $\sigma_{-}$.
Likewise, the Heisenberg pseudodifferential operators are contained in the extended calculus
as those operators for which $\sigma_{0}$ is the pullback to $S^*H\times [-\infty, +\infty]$ of the equatorial symbol of the Heisenberg operator.
Accordingly, the $C^*$-algebra $\Psi_{eH}$ is simply the norm closed subalgebra of bounded operators $\LL(L^2(M))$ generated by $\Psi^0$ and $\Psi_H^0$ (the order zero classical and Heisenberg calculi).
\vskip 6pt

\section{The Fredholm problem for the extended calculus}\label{section:four}

As soon as we have a pseudodifferential calculus on a compact manifold, there arises a Fredholm index problem.
The recipe is familiar.
I'll restrict the discussion to extended Heisenberg operators of order $(0,0)$ (order zero in both the classical and Heisenberg filtrations).
From the existence of the $C^*$-algebraic extension
\[ 0\to \KK\to \Psi_{eH}\stackrel{\sigma_{eH}}{\lra} S_{eH}\to 0,\]
we obtain the usual index theoretic conquences.
An operator $P\in \Psi_{eH}$ with a symbol $\sigma_{eH}(P)$ that is invertible in $S_{eH}$ is a Fredholm operator, and its index depends on the class of $\sigma_{eH}(P)$ in $K_1(S_{eH})$.
In particular, it only depends on the homotopy type of the extended symbol.

Following the general approach to index theorems elaborated by Baum and Douglas in \cite{BD80},  
the index problem for the extended Heisenberg calculus can be framed as follows.
If $P$ is a Fredholm operator in $\Psi_{eH}$,
then all commutators $[P, M_f]$ are compact.
Here $M_f$ denotes the operator on $L^2(M)$ of multiplication with a continuous function $f\in C(M)$.
In the extension above, this follows from the fact that $C(M)$ is central in the continous field $S_{eH}$.
Therefore the operator $P$ defines a class in the Kasparov $K$-homology group
\[ [P] \in KK(C(M),\CC) \cong K_0(M).\]
By general facts in topological $K$-homology it follows that there must be an index formula for $P$ of the form
\[ {\rm Index}\,P = \int {\rm Chern}(?)\wedge {\rm Todd}.\]
While this observation predicts the general form of the index formula, it certainly does not solve it. 
It is a non-trivial problem to compute the {\em topological} cycle (expressed in terms of $K$-theoretic data involving vector bundles etc.) that is to be inserted at the place of the question mark in the above formula .
The ``General index problem'' as formulated in \cite{BD80} is to compute the topological $K$-homology cycle that corresponds to the analytically defined cycle $[P]\in K_0(M)$.  
Equivalently, one can try to identify its Poincar\'e dual in the $K$-theory group $K^0(T^*M)$.
In any case, the problem is to compute such a class in purely {\em topological} terms.

In \citelist{\cite{vE10a} \cite{vE10b}} we presented a solution of this problem for Fredholm operators in the Heisenberg algebra. 
We will now extend those results to the extended Heisenberg algebra.
We achieve this by taking a fresh look at the problem,
introducing new ideas that shed an interesting light on the formula obtained in \cite{vE10b}.
As will become evident in what follows, the solution presented here as well as in \cite{vE10b} for the Heisenberg algebra itself only holds for {\em scalar operators}. 
Surprisingly, it does not seem possible to extend the solution to operators that act on sections in vector bundles! 
This is surprising because in most such cases the introduction of vector bundles merely introduces some extra overhead in the notation, and does not require any truly new ideas or methods.
But as will see, the index problem for {\em scalar} Fredholm operators in the exetended Heisenberg calculus has an elegant topological solution that does {\em not} work for operators acting on vector bundle sections.
\footnote{In this context it is worth pointing out that, while a scalar {\em elliptic} differential operator always has index zero, there {\em are} scalar Heisenberg differential operators with non-zero index.
The standard examples are second order operators of the form $\Delta_H+icT$,
where $\Delta_H$ is a sublaplacian and $T$ the Reeb field.}

\section{The Hermite index theorem}\label{section:five}

One general principle that we learned from our analysis of the index problem for the Heisenberg algebra (as presented, specifically, in \cite{vE10b}) is that the solution of this problem conforms, in its overall structure and in the breakdown of the proof into major steps, to the structure theory of the $C^*$-algebra of order zero operators.

Let us start by considering the Heisenberg algebra.
An obvious decomposition series for the $C^*$-algebra $\Psi_H$ would consist of the two sequences
\begin{align*}
0\to \KK\to &\Psi_H\stackrel{\sigma_{H}}{\lra} S_{H} \to 0,\\
0\to \KK(V^{BF})\oplus \KK(V^{BF}) \to &S_{H}\to C(S^*H) \to 0
\end{align*}
This decomposition series resolves the $C^*$-algebra $\Psi_{H}$ into factors that are Morita equivalent to commutative algebras,
which means that their $K$-theory can be identified with the topological $K$-theory of their spectrum.
For this reason such a resolution is a key ingredient in the solution of the index problem.

The length of the decomposition series (the number of sequences that is needed) is related to the number of Hausdorff strata in the spectrum of $\Psi_{H}$ (the $T_0$ topological space of irreducible unitary representations).
There are three such strata, and therefore two sequences are necessary.
But we can certainly effect the decomposition in a different order, namely as
\begin{align*}
0\to I_H\to &\Psi_H\to C(S^*H) \to 0,\\
0\to \KK\to &I_H\to \KK(V^{BF})\oplus \KK(V^{BF}) \to 0
\end{align*}
The ideal $I_H$ is simply {\em defined} to be the kernel of the equatorial symbol,
and it is referred to as the ideal of {\em Hermite operators}.
It turns out that this way of decomposing $\Psi_H$ is the more fruitful one.

Following this decomposition of the algebra $\Psi_H$, the solution of the index problem proceeds in two steps.
First one solves the problem for the Hermite ideal $I_H$.
One way to achieve this is by reducing the problem for Hermite operators to the Toeplitz index problem solved by Boutet de Monvel \cite{Bo79}.
This reduction is easily achieved using the tools of $K$-homology.
First, let $I_{H, +}$ denote the ideal in $I_H$ consisting of those Hermite operators whose lower hemispherical symbol $\sigma_{-}$ vanishes.
For this ideal we have a slightly simpler sequence 
\[ 0\to \KK\to I_{H, +}\stackrel{\sigma_+}{\lra} \KK(V^{BF}) \to 0\]
There is a natural inclusion of $C(M)$ into the symbol algebra $\KK(V^{BF})$.
This inclusion is constructed by means of the canonical section of ``vacuum vectors'' in the Bargmann-Fok Hilbert bundle $V^{BF}$. 
The Bargmann-Fok space $V^{BF}_x$ at a point $x\in M$ contains a canonical vacuum vector, namely the holomorphic function $1$ on $H^{1,0}$.
Let $s\in \KK(V^{BF})$ denote the family of rank one projectors onto the vacuum.
We then have the natural inclusion
\[ C(M) \to \KK(V^{BF})\;\colon\; f\mapsto fs.\]
As follows from the consideration of the vacuum vectors in \cite{Epxx}, 
the symbols of the form $fs\in \KK(V^{BF})$ corresponds exactly to the symbols of  Toeplitz operators on the contact manifold $M$,
because $s$ is, in fact, the Heisenberg symbol of the (generalized) Szeg\"o projector on $M$.
(The Szeg\"o projector is an order zero pseudodifferential operator in the Heisenberg calculus).
The Toeplitz extension is therefore embedded in the Hermite extension,
\[ \xymatrix{    0 \ar[r] 
               & \KK \ar[r]\ar[d] 
               & \mathcal{T} \ar[r]\ar[d] 
               & C(M) \ar[r]\ar[d] 
               & 0 \\
                 0 \ar[r]
               & \KK \ar[r] 
               & I_{H,+} \ar[r] 
               & \KK(V^{BF}) \ar[r]  
               & 0. }
\]
The inclusion $C(M)\to \KK(V^{BF})$ induces an isomorhism in $K$-homology of the two Morita equivalent $C^*$-algebras.
It is convenient to use the BDF realization of $K$-homology here \cite{Do77}, 
\[ {\rm Ext}(M) \cong {\rm Ext}(\KK(V^{BF})).\]
As is clear from the inclusion of short exact sequences,
the Toeplitz extension $[\mathcal{T}]\in {\rm Ext}(M)$ corresponds, 
under this Morita equivalence,
to the (positive) Hermite ideal $[I_{H,+}]\in {\rm Ext}(\KK(V^{BF}))$.
In this way the Hermite index problem reduces quite naturally and immediately to the Toeplitz index problem.

Taking the Toeplitz index theorem as given,
it follows that if $P$ is a Fredholm operator with $P-1\in I_{H,+}$
then we have
\[ {\rm Index}\,P = \int_M {\rm Ch}(\sigma_{H,+}(P))\wedge {\rm Td}(M).\]
Here $\sigma_{H,+}(P)$ defines a toplogical cycle in $K^1(M)$.
The ``vector bundle'' for this cycle is the Hilbert bundle $V^{BF}$ of Bargmann-Fok spaces,
while the invertible symbol $\sigma_{H,+}(P)$ defines an automorphism of this bundle.
See \cite{vE10b} for a detailed description of this $K$-theory class.

This solves the index problem for the Hermite ideal.
The proof outlined here differs from that presented in \cite{vE10b},
and we feel that the approach taken here greatly simplifies our understanding of the result.
Nevertheless, the method of \citelist{\cite{vE10a} \cite{vE10b}} has the great advantage that it does not rely on Boutet de Monvel's theorem, but rather has that theorem as a corollary.
We would like to point out here that Boutet de Monvel's theorem is proven, originally, for contact manifolds (or, rather, CR manifolds) that arise as pseudoconvex boundaries of complex domains. 
(An elegant proof using relative $K$-homology is found in \cite{BDT89}.)
Our approach in \citelist{\cite{vE10a} \cite{vE10b}} proves the Toeplitz index theorem for {\em arbitrary} (co-oriented) closed contact manifolds.

Nevertheless, our aim here is to emphasize the method of reduction of the index problem for $\Psi_H$ by means of its decomposition series, and in this approach the Toeplitz index theorem emerges as the end point of the reduction
(once we have explained how to reduce the index problem for $\Psi_H$ to $I_H$).
We now turn to the solution of the index problem for the extended Heisenberg algebra, and its reduction to the Hermite index problem.

\section{The index problem in the extended Heisenberg algebra}\label{section:six}

The second step in the solution of the index problem for the Heisenberg algebra presented in \cite{vE10b} involved a reduction of the full problem to the Hermite problem.
We have found a new and more transparant way to achieve this reduction.
As it turns out, our new approach applies to the extended Heisenberg algebra as well as to the Heisenberg algebra proper.
In this section we present the key idea of this paper.
\vskip 6pt

Let $P$ be a bounded Fredholm operator on $L^2(M)$.
If $C\colon L^2(M)\to L^2(M)$ denotes complex conjugation,
we define the {\em transpose} $P^\dagger$ of the operator $P$ as
\[ P^\dagger = C P^* C.\]
Observe that the map $P\mapsto P^\dagger$ is complex linear.
The transpose $P^\dagger$ is Fredholm as well, and
\[ {\rm Index}\,P^\dagger = -\,{\rm Index}\, P.\]
We apply this construction to order zero operators in the extended Heisenberg algebra, and observe the effect of transposition on the symbol.
\begin{lemma}\label{lemma:one}
The extended Heisenberg symbol $\sigma^\dagger = \sigma_{eH}(P^\dagger)$
of the transpose $P^\dagger$ of an order zero operator $P\in \Psi_{eH}$ is related to the symbol $\sigma=\sigma_{eH}(P)$ by the formula
\[ \sigma^\dagger(x,\xi) = \sigma(x,-\xi).\]
Here 
$x\in M$, $\xi\in S^*_{eH}M_x$.
\end{lemma}
\begin{proof}
If $k(x,y)$ is the Schwartz kernel of $P$ then the Schwartz kernel $k^\dagger(x,y)$ of $P^\dagger$ is given by $k^\dagger(x,y)=k(y,x)$.
The proposition now follows immediately by definition of the principal symbol.
\end{proof}

This simple operation can be used to effect a reduction of the index problem of the extended Heisenberg algebra to the Hermite ideal.
Following our general philosophy, we should first study the Type I structure of the $C^*$-algebra $\Psi_{eH}$.
Like the Heisenberg algebra, the extended Heisenberg algebra decomposes according to only two short exact sequences.
To obtain the first sequence,
we consider the `classical region' of the extended symbol.
The classical part of the extended symbol gives rise to the sequence,
\[ 0\to I_H\to \Psi_{eH}\stackrel{\sigma_0}{\lra} C(S^*H\times [-\infty, +\infty])\to 0.\]
The ideal $I_H$ that is the kernel of $\sigma_0$ consists of operators in $\Psi_{eH}$ whose symbol vanishes identically in the classical region.
It follows that these operators are elements in the Heisenberg algebra.
They are, in fact, precisely the Hermite operators.
As our second sequence in the decomposition of $\Psi_{eH}$
we can therefore simply use the same sequence as before,
\[ 0\to\KK\to I_H \to \KK(V^{BF})\oplus \KK(V^{BF})\to 0.\]
Following this structure of the Type I $C^*$-algebra $\Psi_{eH}$ 
and its decomposition by two extensions,
one should aim to reduce the index problem for $\Psi_{eH}$ to that of $I_H$.
This reduction can be achieved quite easily by making clever use of the transpose operation.

Observe that any linear map in the fibers of $T^*M$ induces a diffeomorphism of its extended boundary $S^*_{eH}M$.
Let ${\rm id}_H$ and ${\rm id}_N$  denote the identity maps on $H^*$ and $N^*$ respectively.
Then ${\rm id}_H\oplus -\,{\rm id}_N$ is a linear map on $T^*M$, 
and we can extend this map by continuity to $S^*_{eH}M$.
If $\sigma$ is a function on $S^*_{eH}M$,
let $\sigma^{op}$ denote the function 
\[ \sigma^{op} = \sigma\circ({\rm id}_H\oplus {\rm -\,id_N}).\] 
\begin{lemma}\label{lemma:two}
If $\sigma=\sigma^{0,0}_{eH}(P)$ is the extended Heisenberg symbol of an order zero operator $P\in \Psi_{eH}$,
and $\sigma^\dagger$ is the symbol of its transpose $P^\dagger$,
then $\sigma^\dagger$ is homotopic to $\sigma^{op}$, through a family of invertible symbols in $S_{eH}$.
\end{lemma}
\begin{proof}
The proof relies crucially on the existence of a complex structure in the fibers of the contact hyperplane bundle $H$, compatible with the symplectic form $d\theta$.
With the help of the complex structure $J$ we can form the homotopy of maps
\[ \alpha_t = \cos{\pi t}+J\sin{\pi t},\; t\in [0,1].\]
Because $J$ is compatible with the symplectic form $d\theta$ in each fiber of $H$,
it follows that $\alpha_t$ is an automorphism for the sharp products $\#$ in the two hemispheres of the Heisenberg region in the extended symbol algebra.
(See the remark following the formula for $\#$).
Therefore $\alpha_t$ induces an automorphism of the extended symbol algebra $S_{eH}$.
Observe that $\alpha_0={\rm id}_H$ while $\alpha_1=-\,{\rm id}_H$.

Recall that $\sigma^{op} = \sigma\circ (\,{\rm id}_H\oplus -\,{\rm \,id_N})$, while  $\sigma^\dagger = \sigma\circ (-\,{\rm id}_H\oplus -\,{\rm \,id_N})$ (Lemma \ref{lemma:one}).
Therefore $\sigma_t=\sigma\circ (\alpha_t\oplus -{\rm id}_N)$ provides the required homotopy.
\end{proof}

\begin{corollary}\label{corollary}
If $\sigma\in S_{eH}$ is an invertible extended Heisenberg symbol such that $\sigma^{op}=\sigma$, then the Fredholm index of any operator with symbol $\sigma$ is zero.
\end{corollary}
With these preparations, we can now state our index formula for the extended Heisenberg calculus.
Note that the relative $K$-theory group $K_1(\mathcal{W}(V^{BF}), \KK(V^{BF}))$ is isomorphic to the nonunital $K$-theory $K_1(\KK(V^{BF}))$ (by excision),
which, in turn, is isomorphic to $K_1(C(M))\cong K^1(M)$ (by Morita equivalence).
\begin{theorem}\label{theorem}
Let $(M, H)$ be a closed co-oriented contact manifold,
and $P\in \Psi_{eH}$ a scalar Fredholm operator of order $(0,0)$ in the extended Heisenberg algebra associated to $(M, H)$.
We denote $\sigma_{\pm}=\sigma_{\pm}(P)$.
Then the quotient $\sigma_{+}\,\#_+\, (\sigma_{-})^{-1}$ of the two Heisenberg parts $\sigma_+$ and $\sigma_-$ of the extended symbol (using the $\#_+$ product) defines a $K$-theory class
\[ [\sigma_{+}\,\#\, (\sigma_{-})^{-1}] \in K_1\,(\mathcal{W}(V^{BF}),\, \KK(V^{BF}))\cong K^1(M).\]
Moreover,
\[ {\rm Index}\,P = \int_M {\rm Ch}(\sigma_{+}\#\, [\sigma_{-}]^{-1})\wedge {\rm Td}(M).\]
In particular, the index of $P$ only depends on the values of the extended symbol in the Heisenberg region of $S^*_{eH}M$.
\end{theorem}
\begin{proof}
If $\sigma=\sigma_{eH}(P)$ is the extended Heisenberg symbol of $P$,
then let $\tilde{\sigma}$ denote the function on $S^*_{eH}M$ that agrees with $\sigma$ on the lower half of $S^*_{eH}M$, and that is extended to the upper half in such a way that $\tilde{\sigma}^{op}=\tilde{\sigma}$.
By the ``lower half'' of $S^*_{eH}M$ we mean, of course, the lower half of the classical region $S^*H\times [-\infty, 0]$ together with the lower hemisphere $S^*_{H,-}M$ in the Heisenberg region.

The map $\sigma\mapsto \sigma^{op}$ commutes with the taking of inverses in the symbol algebra $S_{eH}$, because it is an {\em anti-automorphism} of this algebra.
Therefore if $\tilde{\sigma}^{op}=\tilde{\sigma}$ then the same symmetry holds for the inverse $\tilde{\sigma}^{-1}$.
Applying Corollary \ref{corollary} we conclude that the index of $P$ is the same as the index of the quotient $\tau = \sigma \tilde{\sigma}^{-1}$ in $S_{eH}$ (i.e., of an operator in $\Psi_{eH}$ whose extended Heisenberg symbol is $\tau$).

By construction, the function $\tau$ is identically $1$ on the lower half of  $S^*_{eH}M$.
This is true, in particular, for its restriction to the sphere bundle $S^*H\times \{0\}$ at the center of the classical region.
Therefore the restriction of $\tau$ to the corner $S^*H\times \{+\infty\}$
(which is also the boundary of the upper hemisphere in the Heisenberg region)
is {\em homotopic to $1$}.
As an explicit homotopy we can choose the restriction of $\tau$ to $S^*H\times [0,+\infty]$.

The restriction of $\tau$ to the upper hemisphere is precisely $\tau_+ = \sigma_{+}\#(\sigma_{-})^{-1}$. 
From what we have seen, 
the function $\tau_+$ corresponds to an invertible element in the $C^*$-algebra $\mathcal{W}(V^{BF})$ for which the restriction to the boundary $S^*H$ is homotopic to $1$ (as an invertible function in $C(S^*H)$).
Remembering the essential extension
\[ 0\to \KK(V^{BF})\to \mathcal{W}(V^{BF}) \to C(S^*H)\to 0,\]
we see that $\tau_+$ is the representative of a class in the relative $K$-theory group
\[ [\tau_+] \in K_1(\mathcal{W}(V^{BF}), \KK(V^{BF}))\cong K_1(\KK(V^{BF})).\]
The homotopy from the restriction of $\tau$ to the corner $S^*H\times \{+\infty\}$
to its restriction to $S^*H\times \{0\}$
extends to a homotopy of the complete symbol $\tau$
to a symbol $\tau'$ 
that is identically $1$ in the entire classical region $S^*H\times [-\infty,+\infty]$, as well as on the lower hemisphere in the Heisenberg region, like $\tau$ itself.
Therefore $\tau'$ is the extended Heisenberg symbol of an operator in the 
(positive) Hermite ideal $I_{H,+}$, and so we know how to compute its index,
\[ {\rm Index}\,P = \int_M {\rm Ch}(\tau')\wedge {\rm Td}(M).\]
But since $\tau$ and $\tau'$ define the same class in $K^1(M)$,
we can apply the same formula to $\tau_+$,
\[ {\rm Index}\,P = \int_M {\rm Ch}(\tau_+)\wedge {\rm Td}(M).\]
\end{proof}

\vskip 6pt
\noindent {\bf Remark 1.}
The class $[\sigma_{+}\,\#\, (\sigma_{-})^{-1}]$
really defines a {\em topological} cycle in the group $K^1(M)$,
i.e., a complex vector bundle over $M$ together with an automorphism of that bundle.
The vector bundle is of the form
\[ V^N = \bigoplus_{j=0}^N {\rm Sym}^j H^{1,0}\subset V^{BF}\]
for a sufficiently large value of $N$ (depending on the operator $P$).
The automorphism of $V^N$ is obtained from
the family of invertible operators $\sigma_{+}\,\#\, (\sigma_{-})^{-1}$
represented on the Bargmann-Fok Hilbert space $V^{BF}$.
The value of $N$ should be chosen large enough
to ensure that the restriction of $\sigma_{+}\,\#\, (\sigma_{-})^{-1}$
to $V^N$ is still invertible and homotopic to the original element in $K_1( \mathcal{W}(V^{BF}), \KK(V^{BF}))$.
The situation here is exactly the same as for the index formula for the Heisenberg algebra presented in \cite{vE10b}. We refer to that paper for a detailed discussion.

\vskip 6pt
\noindent {\bf Remark 2.}
It is obvious that the formula in \cite{vE10b} for hypoelliptic operators in the Heisenberg algebra is a special case of Theorem \ref{theorem}.
But Theorem \ref{theorem} can also be specialized to {\em elliptic} operators in the classical calculus.
One simply recovers a version of the Atiyah-Singer formula.
In this case $\sigma_{+}$ and $\sigma_{-}$ are simply the restrictions of the classical principal symbol of an elliptic operator $P$ to the two poles in the cosphere bundle $S^*M$,
and we can replace the sharp product $\#_+$ in the formula by a pointwise product.
The quotient $\sigma_+/\sigma_-$ is now, of course, just a map $M\to \CC\setminus \{0\}$, but it still defines a class 
\[ \left[\frac{\sigma_+}{\sigma_{-}}\right]\in K^1(M).\]
Thus, Theorem \ref{theorem} establishes an interesting formal similarity between the hypoelliptic index formula of \cite{vE10b} and the classical elliptic index formula of Atiyah and Singer.

\vskip 6pt
\noindent {\bf Remark 3.}
A strange feature of Theorem \ref{theorem} is that it is not evident how to generalize it to the case of operators that act on sections of vector bundles,
or even to the simpler case of `systems' of (scalar) operators.
To see what goes wrong,
let $P=(P_{ij})$ be a $k\times k$ system of operators in $\Psi_{eH}\otimes M_k\CC$.
The transpose $P^t$ is then the system $(P^t_{ji})$,
and the extended Heisenberg symbol $\sigma^t$ of $P^t$
is given by $\sigma^t(\xi) = \sigma(-\xi)^t$.
A significant difference with the scalar case is that we need to take the {\em matrix transpose} of $\sigma(-\xi)$ here.
It is for this reason that the key trick in the proof of Theorem \ref{theorem} no longer works. 
It is no longer possible to construct an `$op$'-symmetric symbol $\tilde{\sigma}$ that agrees with $\sigma$ on the lower half of the extended sphere (except in the special case where the restriction of $\sigma$ to the sphere $S^*H\times \{0\}$ takes value in symmetric matrices).

We believe that this is a fundamental problem for which there is no simple remedy. Our belief is founded on the role played by the transpose in $K$-homology.
In \cite{BDF77} Brown, Douglas and Fillmore consider the question of finding an explicit determination of inverses in the group ${\rm Ext}(X)=K_1(X)$.
Let $\HH$ be a separable Hilbert space, $\LL(\HH)$ the algebra of bounded operators on $\HH$, $\KK(\HH)$ the compact operators, and $\mathcal{Q}=\LL(\HH)/\KK(\HH)$ the Calkin algebra.
Suppose $\HH$ has a real structure (i.e., a complex conjugation),
so that transposition is defined on $\LL(\HH)$ and hence on $\mathcal{Q}$.
Suppose we have an extension
\[ 0\to \KK(\HH)\to A \to C(X)\to 0.\]
For the Busby invariant $\tau\colon C(X)\to \mathcal{Q}$ associated to this extension one can define the transpose $\tau^t(f) = C\tau(f)^*C$,
which defines a new element $[\tau^t]\in {\rm Ext}(X)$.
Transposition is thus an involution on ${\rm Ext}(X)=K_1(X)$.

Let $P$ be a Fredholm operator in the $C^*$-algebra $A$
with invertible symbol $\sigma\in C(X)$.
The observation that ${\rm Index}\,P^t=-\,{\rm Index}\,P$ can be expressed in terms  of the index pairing of the $K$-theory class $[\sigma]\in K^1(X)$ with the cycles $\tau$ and $\tau^t$,
\[ \ang{[\sigma], [\tau]} = -\,\ang{[\sigma], [\tau^t]}.\]
This may lead one to believe that $[\tau^t]=-[\tau]$ in $K_1(X)$.
However, this is false (See \cite[Remark~8.2]{BDF77}).
One problem is that if we replace $P$ by a {\em system} of operators $P=(P_{ij})$
whose symbol $\sigma$ is an element in $A\otimes M_k\CC$,
then the equality $\ang{[\sigma], [\tau]} = -\,\ang{[\sigma], [\tau^t]}$ no longer holds in general.

We believe that these $K$-homology considerations are intimitely tied up with the role of the operator transpose in our proof of Theorem \ref{theorem},
and that they point to an interesting difference between the index problem for scalar operators and the index problem for systems of operators (or vector bundle operators) in the Heisenberg and extended Heisenberg calculi.
It seems it may be worth trying to understand this better.

\bibliographystyle{amsplain}

\bibliography{MyBibfile}

\begin{bibdiv}
\begin{biblist}

\bib{BD80}{incollection}{
      author={Baum, Paul},
      author={Douglas, Ronald~G.},
       title={{$K$} homology and index theory},
        date={1982},
   booktitle={Operator algebras and applications, {P}art {I} ({K}ingston,
  {O}nt., 1980)},
      series={Proc. Sympos. Pure Math.},
      volume={38},
   publisher={Amer. Math. Soc.},
     address={Providence, R.I.},
       pages={117\ndash 173},
      review={\MR{MR679698 (84d:58075)}},
}

\bib{BDT89}{article}{
      author={Baum, Paul},
      author={Douglas, Ronald~G.},
      author={Taylor, Michael~E.},
       title={Cycles and relative cycles in analytic {$K$}-homology},
        date={1989},
        ISSN={0022-040X},
     journal={J. Differential Geom.},
      volume={30},
      number={3},
       pages={761\ndash 804},
         url={http://projecteuclid.org/getRecord?id=euclid.jdg/1214443829},
      review={\MR{MR1021372 (91b:58244)}},
}

\bib{BG88}{book}{
      author={Beals, Richard},
      author={Greiner, Peter},
       title={Calculus on {H}eisenberg manifolds},
      series={Annals of Mathematics Studies},
   publisher={Princeton University Press},
     address={Princeton, NJ},
        date={1988},
      volume={119},
}

\bib{Bo79}{article}{
      author={Boutet~de Monvel, Louis},
       title={On the index of {T}oeplitz operators of several complex
  variables},
        date={1978/79},
        ISSN={0020-9910},
     journal={Invent. Math.},
      volume={50},
      number={3},
       pages={249\ndash 272},
         url={http://dx.doi.org/10.1007/BF01410080},
}

\bib{BDF77}{article}{
      author={Brown, L.~G.},
      author={Douglas, R.~G.},
      author={Fillmore, P.~A.},
       title={Extensions of {$C\sp*$}-algebras and {$K$}-homology},
        date={1977},
        ISSN={0003-486X},
     journal={Ann. of Math. (2)},
      volume={105},
      number={2},
       pages={265\ndash 324},
      review={\MR{MR0458196 (56 \#16399)}},
}

\bib{CGGP92}{article}{
      author={Christ, Michael},
      author={Geller, Daryl},
      author={G{\l}owacki, Pawe{\l}},
      author={Polin, Larry},
       title={Pseudodifferential operators on groups with dilations},
        date={1992},
        ISSN={0012-7094},
     journal={Duke Math. J.},
      volume={68},
      number={1},
       pages={31\ndash 65},
         url={http://dx.doi.org/10.1215/S0012-7094-92-06802-5},
}

\bib{Do77}{incollection}{
      author={Douglas, Ronald~G.},
       title={Extensions of {$C\sp*$}-algebras and {$K$}-homology},
        date={1977},
   booktitle={{$K$}-theory and operator algebras ({P}roc. {C}onf., {U}niv.
  {G}eorgia, {A}thens, {G}a., 1975)},
   publisher={Springer},
     address={Berlin},
       pages={44\ndash 52. Lecture Notes in Math., Vol. 575},
      review={\MR{MR0474267 (57 \#13914)}},
}

\bib{Epxx}{article}{
      author={Epstein, Charles~L.},
      author={Melrose, R.},
       title={The {H}eisenberg algebra, index theory and homology},
       pages={preprint},
}

\bib{Me97}{incollection}{
      author={Melrose, Richard},
       title={Homology and the {H}eisenberg algebra},
        date={1997},
   booktitle={S\'eminaire sur les \'{E}quations aux {D}\'eriv\'ees
  {P}artielles, 1996--1997},
   publisher={\'Ecole Polytech.},
     address={Palaiseau},
       pages={Exp.\ No.\ XII, 11},
        note={Joint work with C. Epstein and G. Mendoza},
}

\bib{Ta84}{article}{
      author={Taylor, Michael~E.},
       title={Noncommutative microlocal analysis. {I}},
        date={1984},
        ISSN={0065-9266},
     journal={Mem. Amer. Math. Soc.},
      volume={52},
      number={313},
       pages={iv+182},
}

\bib{vE10a}{article}{
      author={van Erp, Erik},
       title={The {A}tiyah-{S}inger formula for subelliptic operators on a
  contact manifold, {P}art {I}},
        date={2010},
     journal={Ann. of Math.},
      number={171},
       pages={1647\ndash 1681},
}

\bib{vE10b}{article}{
      author={van Erp, Erik},
       title={The {A}tiyah-{S}inger formula for subelliptic operators on a
  contact manifold, {P}art {II}},
        date={2010},
     journal={Ann. of Math.},
      number={171},
       pages={1683\ndash 1706},
}

\end{biblist}
\end{bibdiv}

\end{document}